\documentclass[12pt]{article}
\usepackage{etex}
\usepackage{authblk}
\usepackage{amssymb,amsmath,amsthm,epsfig}
\usepackage{latexsym, enumerate}
\usepackage{eepic}
\usepackage{epic}
\usepackage{graphicx}
\usepackage{subfigure}
\usepackage{color}
\usepackage{ifpdf}
\usepackage{dsfont}
\usepackage{multirow}
\usepackage{array}
\usepackage{indentfirst}
\usepackage[colorlinks,linktocpage,linkcolor=blue]{hyperref}
\usepackage[ruled,vlined,linesnumbered]{algorithm2e}
\theoremstyle{plain}

\usepackage{float}

\theoremstyle{definition}

\theoremstyle{remark}

\newcommand{\n}{\noindent}

\newcommand{\bm}{\boldsymbol}

\newtheoremstyle{exampstyle}
  {\topsep} 
  {\topsep} 
  {\itshape} 
  {} 
  {\bfseries} 
  {.} 
  {.5em} 
  {} 

\theoremstyle{exampstyle}

\newtheorem{theorem}{Theorem}

\allowdisplaybreaks
\usepackage{parskip}
\let\oldref\ref
\renewcommand{\ref}[1]{(\oldref{#1})}  
\renewcommand{\eqref}[1]{(\oldref{#1})}

\newbox\boxaddrone \newbox\boxaddrtwo

\def\N+{n\in\mathbb{N}^{+}}

\def\n{\partial{\overrightarrow{\bf n}}}
\def\1d{\mathcal{D}((-\Delta)^{\gamma_1+1/2})}
\def\2d{\mathcal{D}((-\Delta)^{\gamma_2+1})}

\begin{document}

\title{\large\textbf{Restoring the Discontinuous Heat Equation Source Using Sparse Boundary Data and Dynamic Sensors}}
\author[1]{Guang Lin\thanks{guanglin@purdue.edu }}
\author[2]{Na Ou\thanks{oyoungla@csust.edu.cn}}
\author[3]{Zecheng Zhang\thanks{zecheng.zhang.math@gmail.com}}
\author[4]{Zhidong Zhang\thanks{zhangzhidong@mail.sysu.edu.cn}}
\affil[1]{\normalsize{Department of Mathematics and School of Mechanical Engineering, Purdue University, West Lafayette, IN}}
\affil[2]{\normalsize{School of Mathematics and Statistics, Changsha University of Science and Technology, China}}
\affil[3]{\normalsize{Department of Mathematics, Florida State University, Tallahassee, FL}}
\affil[4]{\normalsize{School of Mathematics (Zhuhai), Sun Yat-sen University, China}}

\date{}
\maketitle

\begin{abstract}
\noindent This study focuses on addressing the inverse source problem associated with the parabolic equation. We rely on sparse boundary flux data as our measurements, which are acquired from a restricted section of the boundary.
While it has been established that utilizing sparse boundary flux data can enable source recovery, the presence of a limited number of observation sensors poses a challenge for accurately tracing the inverse quantity of interest. To overcome this limitation, we introduce a sampling algorithm grounded in Langevin dynamics that incorporates dynamic sensors to capture the flux information.
Furthermore, we propose and discuss two distinct sensor migration strategies. Remarkably, our findings demonstrate that even with only two observation sensors at our disposal, it remains feasible to successfully reconstruct the high-dimensional unknown parameters.
\\


\end{abstract}

\section{Introduction}
\subsection{Mathematical model}
We give the mathematical statement of our interested inverse problem. Firstly the heat equation is given as follows:   \begin{equation}\label{PDE}
 \begin{cases}
  \begin{aligned}
    (\partial_t-\Delta)u(x,t)&=\chi_{{}_{D}}, &&(x,t)\in\Omega\times(0,T),\\
    u(x,t)&=0,&&(x,t)\in\partial\Omega\times(0,T)\cup \Omega\times\{0\}.
  \end{aligned}
  \end{cases}
 \end{equation} 
Here $\Omega\subset \mathbb R^2$ is the two-dimensional unit disc. In equation \eqref{PDE}, the discontinuous source $\chi_{{}_{D}}$ is the characteristic function and the support $D\subset \Omega$ is unknown. More precisely, if $x$ is included by $D$, the value of $\chi_{{}_{D}}$ is one; otherwise, it equals zero. We assume that $\partial D$ is sufficiently smooth.   
In this work, we will use the boundary flux data observed by the dynamic moving sensors to recover the unknown support $D$. The mathematical formulation of the measurements can be written as  
 \begin{equation*}
 \frac{\partial u}{\n}(z(t),t),\ z(t)\in \Gamma(t)\subset\partial\Omega,\ t\in(0,T).
 \end{equation*}
 Here $z(t)$ is the observation point and located in the boundary $\partial\Omega$; $\Gamma(t)$ is the observed area. We employ the notations $z(t)$ and $\Gamma(t)$ to emphasize that the sensor locations are time-dependent. Our analysis centers on sparse boundary data, where $\Gamma$ represents a very small subset of the boundary. The limited availability of observations on this sparse boundary presents substantial challenges in recovering $\chi_D$. To address these challenges and verify the proposed Algorithm \ref{algo_pseudo_algo}, we opt to rely solely on flux information obtained from two sensors, further enhancing the complexity of the problem.

\subsection{Background and literature}
In practical applications of inverse problems, sparse boundary data holds great importance. Consider Equation \eqref{PDE}, which can describe the diffusion of pollutants in a system, where the solution $u$ represents pollutant concentration. In such scenarios, the region of the pollutant source, denoted as $D$, should be heavily contaminated. To safeguard the health of engineers, the utilization of boundary measurements becomes imperative.

Furthermore, in the realm of practical inverse problems, the project's cost cannot be overlooked. This includes expenses related to equipment, labor, computation, and more. This cost-conscious perspective has motivated numerous researchers to seek solutions to inverse problems using sparse data.
 
However, when opting for sparse boundary data, it raises a pertinent question: how do we determine the sensors' location? Although one can establish the well-posedness of the inverse PDE problem, the sensors' locations significantly influence the quality of numerical simulation. Unfortunately, in the context of inverse problems, we often have to select observation areas randomly. In practice, it is entirely possible to initially choose suboptimal observation points, which can pose challenges.

To address this dilemma and ensure the quality of our approximations, we propose the concept of using moving observation points. Specifically, if the initial observation points are situated in less favorable locations, we can relocate the facilities to more suitable positions. Consequently, the data obtained from these improved locations will compensate for the adverse effects of the suboptimal points and enhance the accuracy of our approximations.

In line with this idea, the fundamental question we must address is how to determine whether a location is advantageous or not; that is, how to decide when and where to move the observation points.

The inverse source problem of the heat equation $(\partial_t-\Delta)u=F$ is a classical field in the literature of inverse problems, and abundant academic achievements are generated. See \cite{ChengLiu:2020,HuangImanuvilovYamamoto:2020,RundellZhang:2018,RundellZhang:2020,JinKianZhou:2021,Ikehata:2007,Isakov:1990,HettlichRundell:2001, LinZhangZhang} and the references therein. Nowadays, due to the significant application value of the sparse data, more and more researchers have paid attention to this field. Here we list several references on the inverse problems with sparse data. In \cite{HettlichRundell:2001}, the authors consider the inverse source problem of the heat equation on the two-dimensional unit disc. They recover the space-dependent source $f(x)$ by the flux data at finite points of the boundary. The conclusion in \cite{HettlichRundell:2001} is promoted by \cite{RundellZhang:2020}, in which the variable separable source $f_1(x)f_2(t)$ can be uniquely determined by the sparse boundary data. Note that in \cite{RundellZhang:2020}, the spatial term $f_1(x)$ and temporal term $f_2(t)$ are both unknown. In \cite{LiZhang:2020, LinZhangZhang}, the authors further recover the semi-discrete unknown source, which is investigated in this work, and the considered equation is the parabolic equation on a general domain. Also, for the geometric inverse problems, the authors in \cite{KianLiLiuYamamto:2021, HelinLassasYlinenZhang:2020} recover the manifold by the image of a single specific function. We should note that for such geometric inverse problems, people usually use a whole operator as the measurements.

\subsection{Main result and outline}
We summarize the contributions as follows.
\begin{enumerate}
    \item We study inverse problems with sparse measurements and moving observation points.  
    \item We introduce a sampling-based approach and a dynamic sensor migration algorithm designed for source tracing. Moreover, we present two migration strategies and validate their effectiveness through two challenging numerical examples.
\end{enumerate}

The rest of this article is organized as follows. In Section \ref{sec_prelim}, we will provide an overview of the existing problems. We then introduce the Bayesian methodology and the Langevin-based algorithm in Section \ref{sec_bayesian_method}, in which the uniqueness of this inverse problem is also discussed. In particular, we will discuss the proposed sampling methods and proposed dynamical sensors migration algorithm in Section \ref{sec_adaptive_pcn} and \ref{sec_real_algo}. Finally, we will verify the performance and introduce the sensors migration strategies in Section \ref{sec_numerical}.

\section{Preliminaries}
\label{sec_prelim}

 \subsection{The setting of the inverse problem}
 
This article is an extension of reference \cite{RundellZhang:2018, RundellZhang:2020}, in which the authors prove the uniqueness theorem of the inverse source problem under sparse data. More precisely, under suitable conditions, the flux data generated from two points of the boundary can uniquely determine the unknown source. However, as we mention in the introduction, when we solve this inverse problem numerically, there is a natural question that how to determine the locations of the observation points. This question holds paramount importance in computational work because the precision of the numerical approximation is profoundly influenced by the placement of observation points. To substantiate this assertion, we test a numerical experiment by the algorithm proposed in \cite{RundellZhang:2020} and obtain the comparison (please see Figure \ref{fig_comparison}). 

\begin{figure}[th!]
\center
\subfigure{
\includegraphics[trim = .5cm .5cm .5cm .5cm, clip=true,width=6cm,height=5.7cm]
{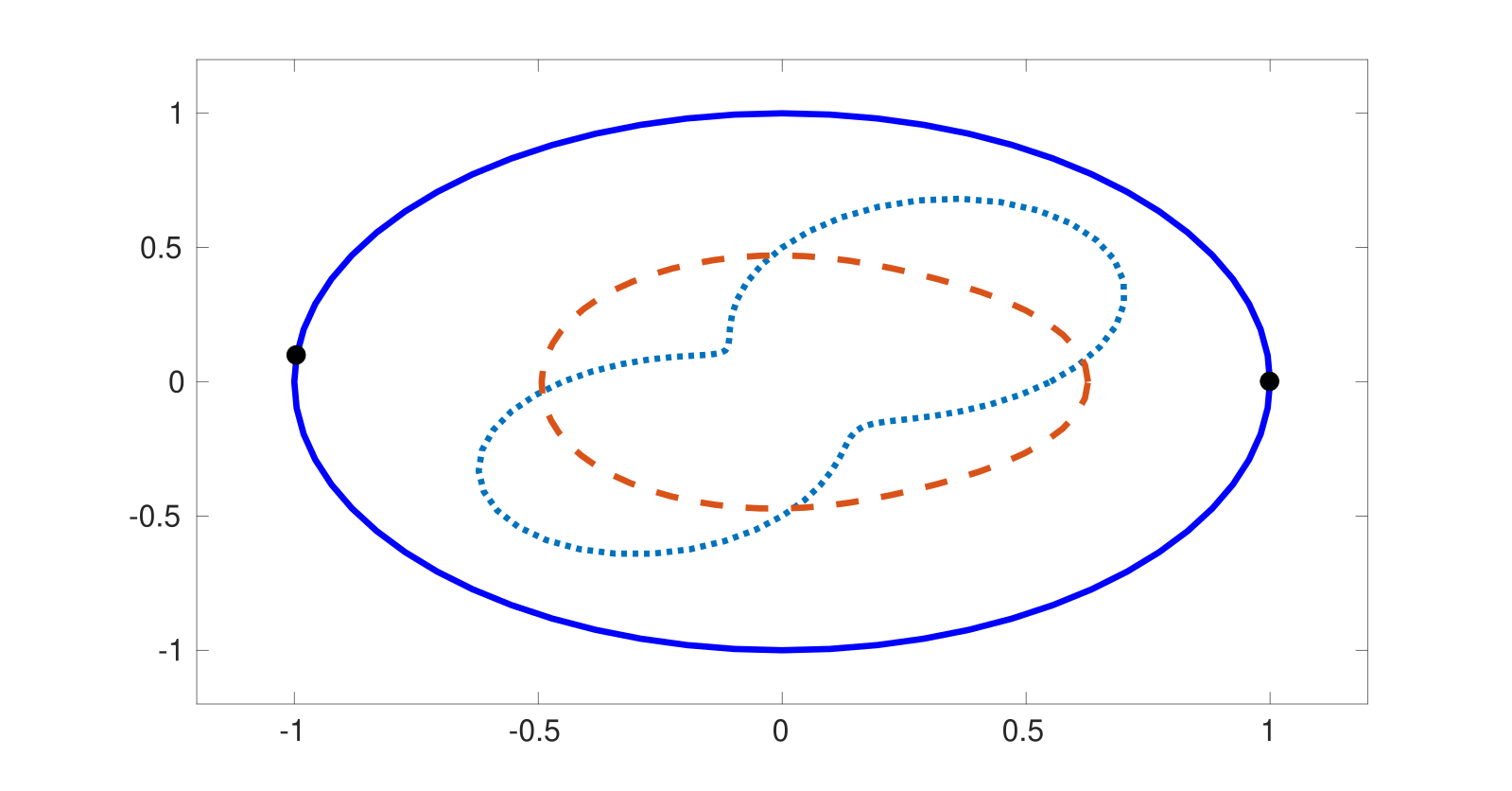}}
\subfigure{
\includegraphics[trim = .5cm .5cm .5cm .5cm, clip=true,width=6cm,height=5.7cm]
{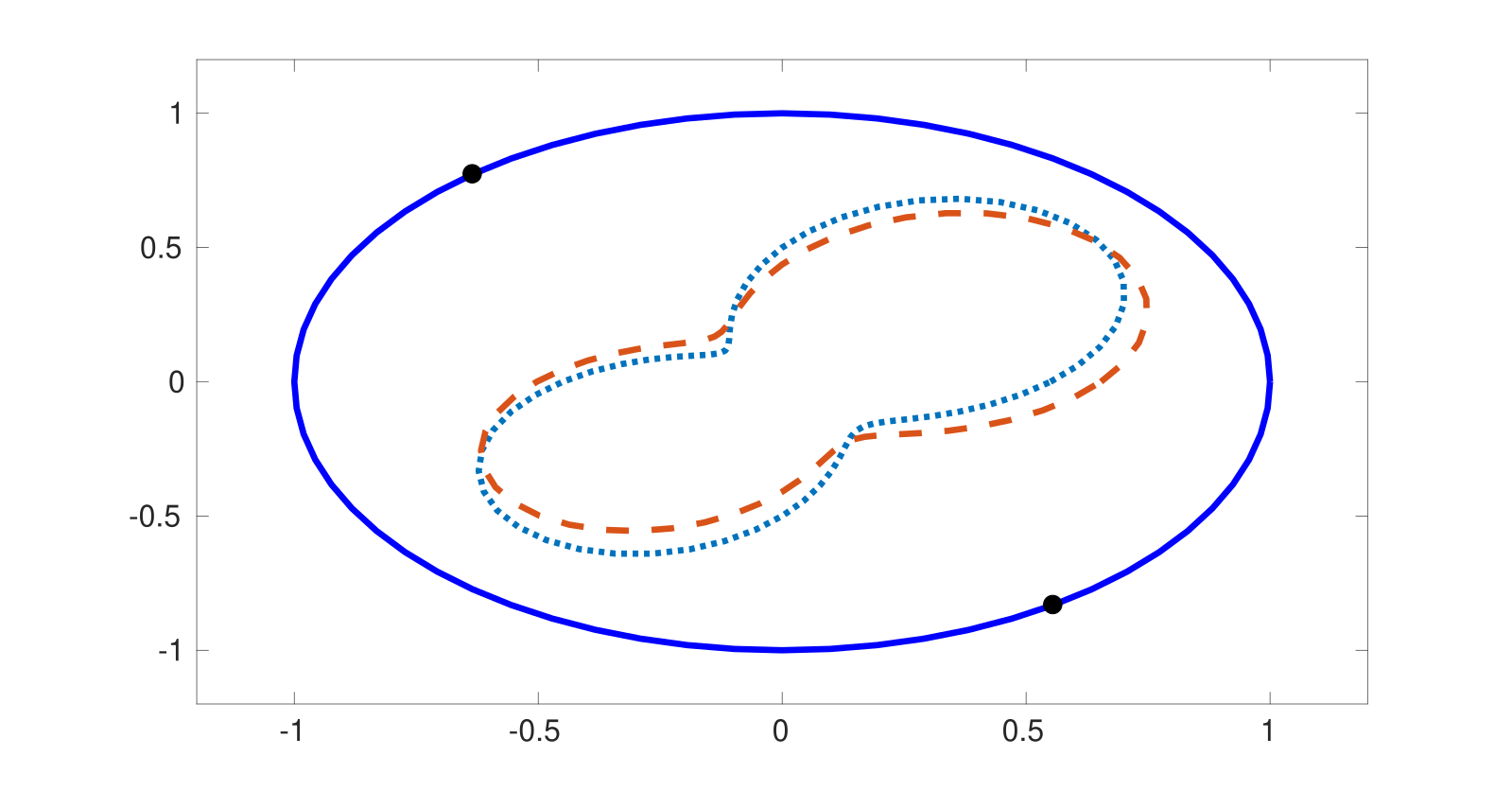}}
\caption{Numerical approximations with different observation points. 
The target source location is represented by the blue dashed curve, while the brown curve depicts the source recovered using only two sensors. The black dots situated along the boundary mark the positions of these sensors. It is evident from the images that the accuracy of the recovery process is contingent upon the placement of the sensors. }
\label{fig_comparison}
\end{figure}

In Figure \ref{fig_comparison}, the dotted blue line represents the support of the exact source, while the dashed red line corresponds to the support of the approximated source. The black points located along the boundary represent the observation points. It's evident that the support of the exact source in both figures is the same; the only difference lies in the placement of the observation points. Consequently, the two approximations exhibit significant disparities.

Evidently, in the left figure, the observation points are positioned unfavorably, earning them the label of `bad' points, while those in the right figure are considered `good'. However, it's important to recognize that in practical applications of inverse problems, obtaining the exact solution for the unknown source is often infeasible. Therefore, we are left with the task of selecting observation points largely based on chance.

If, by chance, we initially select `bad' points and do not alter their positions, it results in the production of poor numerical outcomes, as demonstrated in the left part of Figure \ref{fig_comparison}. This underscores the reason why we endeavor to leverage data from moving observation points to offset the approximations adversely affected by these 'bad' observation points.

Following \cite{HettlichRundell:2001, RundellZhang:2020}, the unknown support $D$ of the source can be uniquely determined when the observed area $\Gamma$ is fixed and only consists of two chosen points on the boundary. Hence, in this work, we will suppose that $\Gamma(t)$ consists of two points of $\partial \Omega$ for each $t\in(0,T)$. With the idea of moving observation points, the measurements we used in this work will be represented as  
  \begin{equation}\label{data}
  \frac{\partial u}{\n}(z_j(t),t),\ z_j(t)\in \partial\Omega,\ j=1,2,\ t\in (0,T).
 \end{equation}
 The notation $z_j(t)$ is the orbit of each observation point, and the essential difficulty of this work is how to determine the orbit $z_j$. We will provide some rules for determining $z_j(t)$ and give the corresponding numerical experiments. These will be discussed in the next sections.

\subsection{Bayesian framework}
A highly effective approach for tackling the inverse problem involves adopting the Bayesian inference framework, as referenced in prior works \cite{stuart2010inverse, lin2023b, efendiev2006preconditioning,chung2020multi,zhang2023bayesian}. In this context, we define the dataset as ${\bf d} = \{d_i\}_{i = 1}^{n_d}$ and represent the prior distribution of $\bm \xi$ as $p(\bm \xi)$. Utilizing the Bayesian formula, we can express this relationship as follows:
\begin{align*}
    p(\bm \xi|{\bf d}) \propto p(\bm \xi) p({\bf d}|\bm \xi).
\end{align*}
where $p(\cdot|\bm \xi)$ is a likelihood function given the parameter $\bm \xi$. 
where $p(\cdot|\bm \xi)$ represents a likelihood function conditioned on the parameter $\bm \xi$.

We denote the mapping from the target parameters $\bm \xi$ to the observations as the forward model, labeled as $g$. Given that observation data inherently contain noise and that noise is also introduced during the forward evaluation, we make the assumption that the likelihood follows a normal distribution with a mean of $\bf d$ and a variance of $\sigma^2$. Specifically, the likelihood function takes the following form:
\begin{align*}
    p({\bf d}|\bm \xi)=(2\pi\sigma^2)^{-\frac{n_d}{2}}\exp\bigg(-\frac{\|{\bf d}-g(\bm \xi)\|_2^2}{2\sigma^2}\bigg).
\end{align*}
One of the key advantages of employing the Bayesian inference method is our ability to obtain the distribution of the unknown parameters. This capability allows us to quantitatively assess the uncertainty associated with quantities of interest (QoI), such as the outputs of the forward model. The trajectory of the observation points is subsequently determined in a sequential manner using posterior samples of the QoI.

\section{Langevin diffusion and Bayesian sampling}
\label{sec_bayesian_method}
{In this study, we will employ the Bayesian approach to ascertain the trajectory of observation points and iteratively address the inverse source problem. To attain the desired posterior distribution $p(\bm \xi|{\bf d})$, various MCMC methods can be employed for sampling. Among these methods, the Langevin diffusion-based MCMC \cite{dalalyan2017further, dalalyan2019user, raginsky2017non, chen2014stochastic} stands out as a computationally efficient choice, particularly when dealing with high-dimensional and multimodal distributions \cite{raginsky2017non, li2023fast}.
Specifically,
let us define the energy function $U(\bm \xi)$ as,
\begin{align*}
    U(\bm \xi) = \log p(\bm \xi)+\sum_{i = 1}^{n_d} \log p(d_i|\bm \xi).
\end{align*}
We are interested in sampling from the following distribution:
\begin{align*}
    \pi(\bm \xi) \propto \exp(-\frac{U(\bm \xi)}{\tau}),
\end{align*}
where $\bm \xi \in \mathbb{R}^p$, $\tau$ is the temperature parameter, and the energy function $U$ is known. One method for generating samples from $\pi (\bm \xi)$ is to simulate a stochastic process whose stationary distribution is $\pi$. Under certain regularity conditions (see, e.g., \cite{bhattacharya1978criteria, roberts1996exponential}), the following stochastic differential equation (SDE), known as the preconditioned Langevin diffusion (LD) \cite{roberts2002langevin,nguyen2019nonasymptotic}, has a stationary distribution $\pi$:
\begin{equation}
    \mathrm{d}\bm \xi_t = -L\nabla U (\bm \xi_t) \mathrm{d}t + \sqrt{2 L \tau} \mathrm{d}W_t,
    \label{eqn_langevin}
\end{equation}
where $L$ is an arbitrary symmetric, positive-definite matrix, $W_t$ is the $p$-dimensional standard Brownian motion. Motivated by the replica-exchange methods \cite{lin2022multi, lin2023b}, \cite{na2022replica} extends the preconditioned Langevin diffusion-based methods to multiple chains to accelerate distribution simulation. 

In practical sampling scenarios, computing the gradient of the energy function can be computationally expensive, especially when dealing with PDE-based inverse problems that involve time-consuming forward model simulations. Moreover, in this work, the source term exhibits discontinuities concerning the unknown parameters. Therefore, we opt for the gradient-free Metropolis-Hasting (MH) MCMC method \cite{2004Markov, gamerman2006markov}.

To address the inverse problem, we will employ the adaptive preconditioned Crank-Nicolson MCMC (pCN-MCMC) method within a Bayesian inference framework.

\subsection{Adaptive preconditioned-Crank–Nicolson (pCN) }
\label{sec_adaptive_pcn}
As the sampling acceptance rate is independent of the dimension of parameters, the method of pCN-MCMC \cite{cotter2013mcmc} stands as an efficient strategy for addressing large-scale inference challenges. Specifically, it is based on the SDE
\begin{equation}
\mathrm{d}\bm \xi_t = -LB^{-1} \bm \xi_t \mathrm{d}t + \sqrt{2 L} \mathrm{d}W_t,
    \label{eqn_sde}
\end{equation}
which is a special case of equation (\ref{eqn_langevin}), where the prior is the Gaussian distribution ${\cal N}(0, B)$, the gradient of the log-likelihood vanishes in the drift term, and the temperature is $\tau=1$. The proposal is generated by applying the Crank–Nicolson (CN) schemes \cite{beskos2008mcmc, cotter2013mcmc} to the SDE (\ref{eqn_sde}),
\begin{equation}
\bm \xi^*=\bm \xi-\frac{1}{2}\delta LB^{-1}(\bm \xi+\bm \xi^*)+\sqrt{2\delta L}w,
    \label{eqn_CN}
\end{equation}
where $w \sim \mathcal{N}(0, I_p)$ represents samples from a multivariate normal distribution, and $\delta$ denotes the step size. When $L=B$ is chosen, the pCN proposal is resulted,
\begin{equation}
\bm \xi^*=\sqrt{1-\beta_1^2}\bm \xi+\beta_1 w, \ w\sim \mathcal{N}(0, B),
    \label{eqn_pCN}
\end{equation}
where 
\[
\beta_1=\frac{2\sqrt{2 \delta}}{2+ \delta}.
\]

Moreover, the equation (\ref{eqn_CN}) can also be rewritten as \cite{hu2017adaptive} 
\begin{equation}
\bm \xi^*=\sqrt{1-\beta_2^2CB^{-1}}\bm \xi+\beta_2 w, \ w\sim \mathcal{N}(0, C),
    \label{eqn_apCN}
\end{equation}
where 
\[
\beta_2\sqrt{C}=\frac{\sqrt{2 \delta}}{I+ \frac{1}{2}\delta L B^{-1}},
\]
yielding the adaptive pCN scheme, it is used to accelerate the convergence of the pCN MCMC method. One of the optimal choices of $C$ is the posterior covariance matrix, which is not directly available but can be determined by the historical samples empirically. For both the pCN and adaptive pCN scheme, the acceptance rate has the form
\begin{equation}
\alpha(\bm \xi^*, \bm \xi)=\min\left(1, \exp\left(\Phi(\bm \xi)-\Phi(\bm \xi^*)\right)\right).  
\label{acc-rate}
\end{equation}
where
\[
\Phi(\bm \xi)=\log p({\bf d}|\bm \xi),
\] 
please refer to Algorithm \ref{algo_ApCN} for more details on the adaptive pCN sampling method.

\subsection{Methodology}
\label{sec_real_algo}

As mentioned earlier, the uniqueness of this inverse problem has undergone comprehensive scrutiny in prior works \cite{RundellZhang:2018, RundellZhang:2020}. In this study, we align ourselves with the principles outlined in these references \cite{RundellZhang:2018, RundellZhang:2020} to articulate the uniqueness theorem.

\begin{theorem}\label{uniqueness}
Given two unknown supports $D_1$ and $D_2$ which have sufficiently smooth boundaries, we denote the solutions corresponding to $D_j$ by $u_j$, $j=1,2$. Recalling the data \eqref{data}, we assume that there exists $\epsilon_0>0$ such that 
$$z_j(t)=\theta_j\text{\ for\ } t\in(0,\epsilon_0),\ j=1,2,  \text{\ and\ } \theta_1-\theta_2\notin \pi \mathbb Q,$$ 
where $\mathbb Q$ is the set of rational numbers. If 
$$\frac{\partial u_1}{\n}(z_j(t),t)=\frac{\partial u_2}{\n}(z_j(t),t),\ j=1,2,\ t\in (0,T),$$
then $D_1=D_2$ in the sense of $L^2(\Omega)$. 
\label{thm_main}
\end{theorem}
\begin{proof}
The proof of Theorem \ref{uniqueness} follows from \cite[Theorem 3.1]{RundellZhang:2018} and the analyticity of the natural exponential function straightforwardly.   
\end{proof}

Theorem \ref{uniqueness} establishes that, under appropriate conditions, we can ascertain the precise support of the source using the data \eqref{data}. However, extensive experimental findings, as presented in \cite{LinZhangZhang, RundellZhang:2020}, suggest that relying solely on fixed-sensor flux observations may not yield accurate source predictions.
Furthermore, in both sampling methods \cite{LinZhangZhang, stuart2010inverse} and the classical Levenberg–Marquardt iteration, the accuracy of source capture is greatly contingent on sensor placement, particularly when only using a limited number of sensors. 

Inspired by Theorem \ref{uniqueness}, we introduce an algorithm that introduces variability in the locations of observation points (sensors). For a comprehensive understanding of the procedure, we direct your attention to Algorithms \ref{algo_pseudo_algo} and \ref{algo_ApCN}. It's worth noting that the option to select new sensor locations at random is a possibility. However, in Section \ref{sec_sensor_moving_circle} and \ref{sec_sensor_moving_peanut}, we delve into two distinct approaches for altering the sensor locations, which we detail further.

\begin{algorithm}[h!]
\caption{Peudo-algorithm for dynamical locations determination.}
    \label{algo_pseudo_algo}
Initialization: two observation points (sensors) $p_1$ and $p_2$, constant $n$, observation time stamps $t' = [T_1,...,T_{n}]^{\intercal}$,  current observation time $\tilde{t} = T_1$, initial PDE simulation time $t_0 = 0$, \textcolor{black}{ and the total sampling iteration limits $K$};

\For{$k = 1$ to $K-1$}{
Observe the flux at $p_1$ and $p_2$\;
Simulate the PDE solution until $\Tilde{t}$ from $t_0$ for the computation of $\Phi(\cdot)$\;
Obtain posterior samples $\{\bm \xi_i\}_{i=1}^{N}$ and $\{U_i\}_{i=1}^{N}$ according to the sampling schemes presented in Algorithm \ref{algo_ApCN} \;
Change two observation points $p_1$ and $p_2$\ basing on posterior samples \;
Set initial time $t_0 = \tilde{t}$\;
Set the mean of $\{U_i\}_{i=1}^{N}$ to be the initial value of the PDE \;
Pick up and set $\tilde{t}$ to be the next observation time $T_{k+1}$ from $t'$\;
    }

\end{algorithm}

\begin{algorithm}[h!]
	\caption{The adaptive pCN MCMC.}
	\label{algo_ApCN}
    \KwIn{The covariance matrix $B$ of the prior, parameters $\beta_1$, $\beta_2$, noise of the observed data $\sigma$, the initial guess $\bm \xi_0$, current observation time $\tilde{t}$. The number of iterations $N_1$ and $N$, the update frequent number $k_0$. }
    \KwOut{Posterior samples $\{\bm \xi_i\}_{i=1}^{N}$ and the PDE solution $\{U_i\}_{i=1}^{N}$ at time $\bar t$. }
    
	\BlankLine

    \For {$k=1:N_1$}{

       Generate the candidate $\bm \xi^*$ by the pCN scheme (\ref{eqn_pCN})

       Accept $\bm \xi^*$ as $\bm \xi_{k}$ with the rate $\alpha(\bm \xi^*, \bm \xi_{k-1})$

       }
     Collect the posterior samples of the parameters $\{\bm \xi_i\}_{i=1}^{N_1}$ and the PDE solution $\{U_i\}_{i=1}^{N_1}$ at time $\bar t$.
     
    Compute the empirical covariance matrix of the $\{\bm \xi_i\}_{i=1}^{N_1}$ as $C$.

    \For {$k=N_1+1:N$}{

       \If {$mod(k,k_0+1)=0$} {

       Update the empirical covariance matrix of $\{\bm \xi_i\}_{i=1}^{k}$ as $C$

       }

       Propose the candidate $\bm \xi^*$ by the adaptive pCN scheme (\ref{eqn_apCN})

       Accept $\bm \xi^*$ as $\bm \xi_{k}$ with the probability $\alpha(\bm \xi^*, \bm \xi_{k-1})$

       }
     Store the posterior samples of the parameters and PDE solutions at time $\bar t$.
\end{algorithm}

\section{Numerical experiments.}
\label{sec_numerical}
In this section, we will conduct numerical experiments to test the performance of the algorithm.
We shall delve into scenarios involving two distinct unknown sources, each characterized by disparate geometries and dimensions pertaining to the unknown parameters. Across both sets of problems, the task involves measuring boundary observation flux at two distinct points along the boundary—an endeavor that has proven to be particularly challenging for conventional implementation methodologies \cite{RundellZhang:2020, RundellZhang:2018}. Employing a comparative approach alongside the fixed observation points method, our analysis reveals a consistent trend: the methods we propose adeptly capture the source.

Consider the problem
\begin{equation*}
\frac{\partial u}{\partial t}-\Delta u= b\chi_D(x), \ \  \ t\in [0, T]
\end{equation*}
the physical domain is $\Omega:=\{(x, y)|x^2+y^2<1\}$. $b$ represents the predefined intensity of the source term. $\chi_D$ denotes the characteristic function related to the domain $D$, where $D$ is the unspecified region in need of reconstruction.
We will specify $D$ in the experiments later.
The observations are the flux $\frac{\partial u}{\partial \nu}$ collected at some time and locations along $\partial \Omega$. 

For the convenience of solving the forward model, we transform the Cartesian plane to the polar coordinates, i.e.,
\[
r=\sqrt{x^2+y^2},\ \ \theta=\tan{\frac{y}{x}},
\]
the domain is then transformed to be $\tilde \Omega:=\{(r, \theta)|0<r<1, 0 \leq \theta\leq 2\pi\}$. The equation becomes
\begin{equation}\label{rthetaEq}
\frac{\partial u}{\partial t}-\left(\frac{1}{r}\frac{\partial }{\partial r} (r \frac{\partial u}{\partial r})+\frac{1}{r^2}\frac{\partial^2 u}{\partial \theta^2}\right)= b\chi_D(r, \theta), t\in [0, T].
\end{equation}
The forward model (\ref{rthetaEq}) is solved by the finite difference method. The measurement locations are sequentially determined during the inference following Algorithm \ref{algo_pseudo_algo}.

\subsection{Circle-shape source}
In this experiment, we consider the following unknown domain $D$,
\[
(x-\eta_1)^2+(y-\eta_2)^2<0.2^2,
\]
where ${\bm \eta}=[\eta_1, \eta_2]^{\intercal}$ are the unknowns. The corresponding expression under the polar coordinates is
\[
(r\cos{\theta}-\rho\cos{\omega})^2+(r\sin{\theta}-\rho\cos{\omega})^2<0.2^2,
\]
where $(\rho, \omega)$ are the unknown parameters. As the range of the samples is $(-\infty, +\infty)$ for this algorithm, while the range of our parameters is finite, we use the bijection map
\begin{eqnarray*}
\rho(\xi_1) &=& \frac{1}{\pi}\arctan{\xi_1}+\frac{1}{2}, \\
  \omega( \xi_2) &=& 2\arctan{\xi_2}+\pi
\end{eqnarray*}
to map ${\bm \xi}=(\xi_1, \xi_2)^{\intercal}$ to the intervals $(0, 1)$ and $(0,2\pi)$, respectively. The ${\bm \xi}$ become the unknowns endowed with the prior $\mathcal{N}(0, I_p)$ to be inferred. The relationship between ${\bm \xi}$ and the location of the center is
\begin{eqnarray*}
\eta_1 &=& \rho(\xi_1)\cos{\omega(\xi_2)}, \\
\eta_2 &=& \rho(\xi_1)\sin{\omega(\xi_2)}.
\end{eqnarray*}
The statistical properties of ${\bm \eta}$ can be estimated through the posterior samples of ${\bm \xi}$.

The forward model is computed on a spatial grid of dimensions $33\times 36$, with a predefined strength of $b=50$. The observed data is generated using ground truth parameters ${\bm \eta}=[0, 0.5]^{\intercal}$, implying that the true parameter vector $\bm \xi$ is given by $(\xi_1,  \xi_2)^{\intercal}=(0, -1)^{\intercal}$. The measurement error variance is $\sigma^2=0.05^2$. We perform a total of $N_1=0$ initial iterations, followed by $N=10^4$ iterations, with updates occurring every $k_0=2.5\times 10^3$ iterations. In both the pCN and adaptive pCN schemes, the tunable parameters $\beta_1$ and $\beta_2$ are set to be equal to ensure an acceptance rate of 30\% to 40\%.

\begin{figure}[h!]
  \centering
  \subfigure[]{
  \includegraphics[width=0.31\textwidth,height=1.9in]{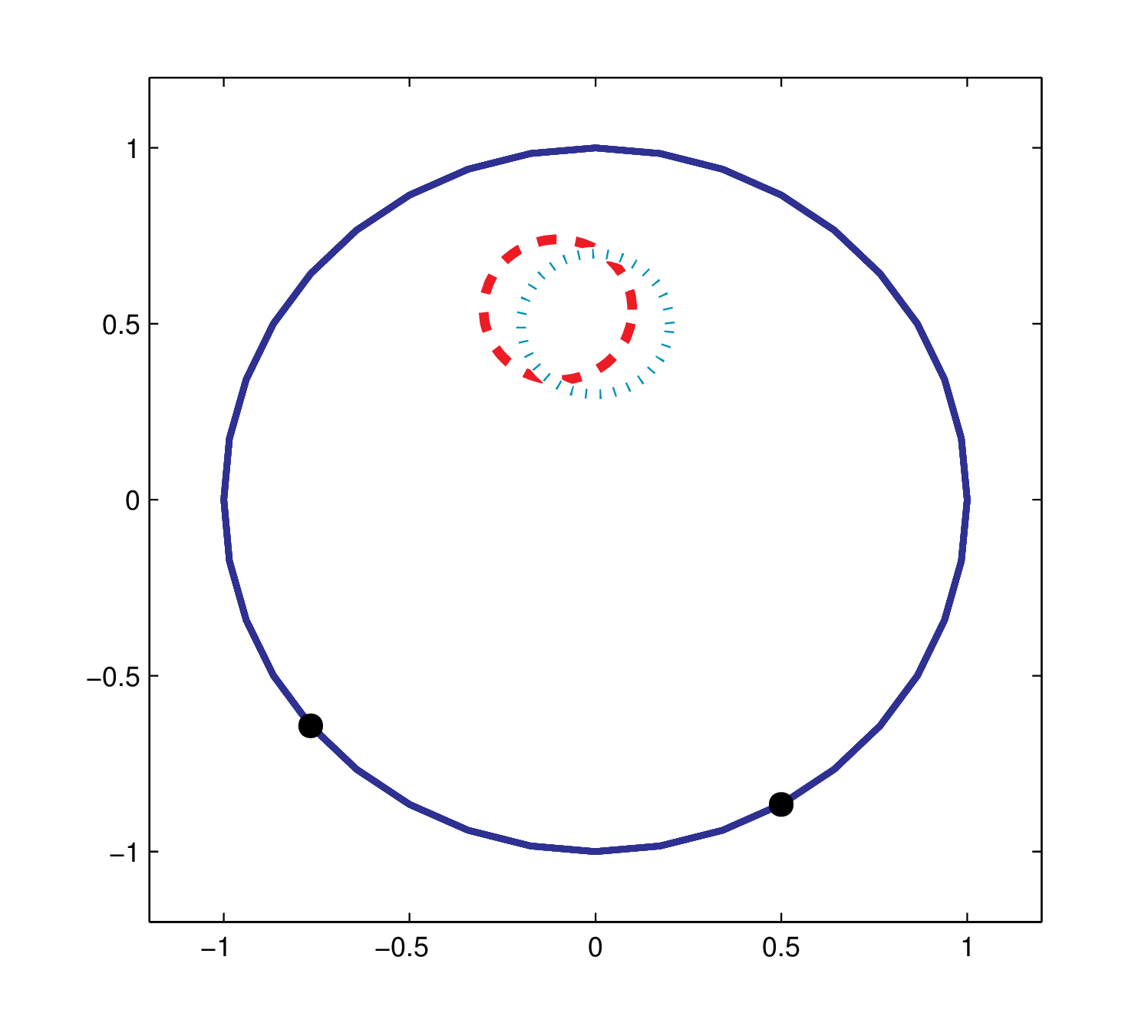}}
  \subfigure[]{
  \includegraphics[width=0.31\textwidth,height=1.9in]{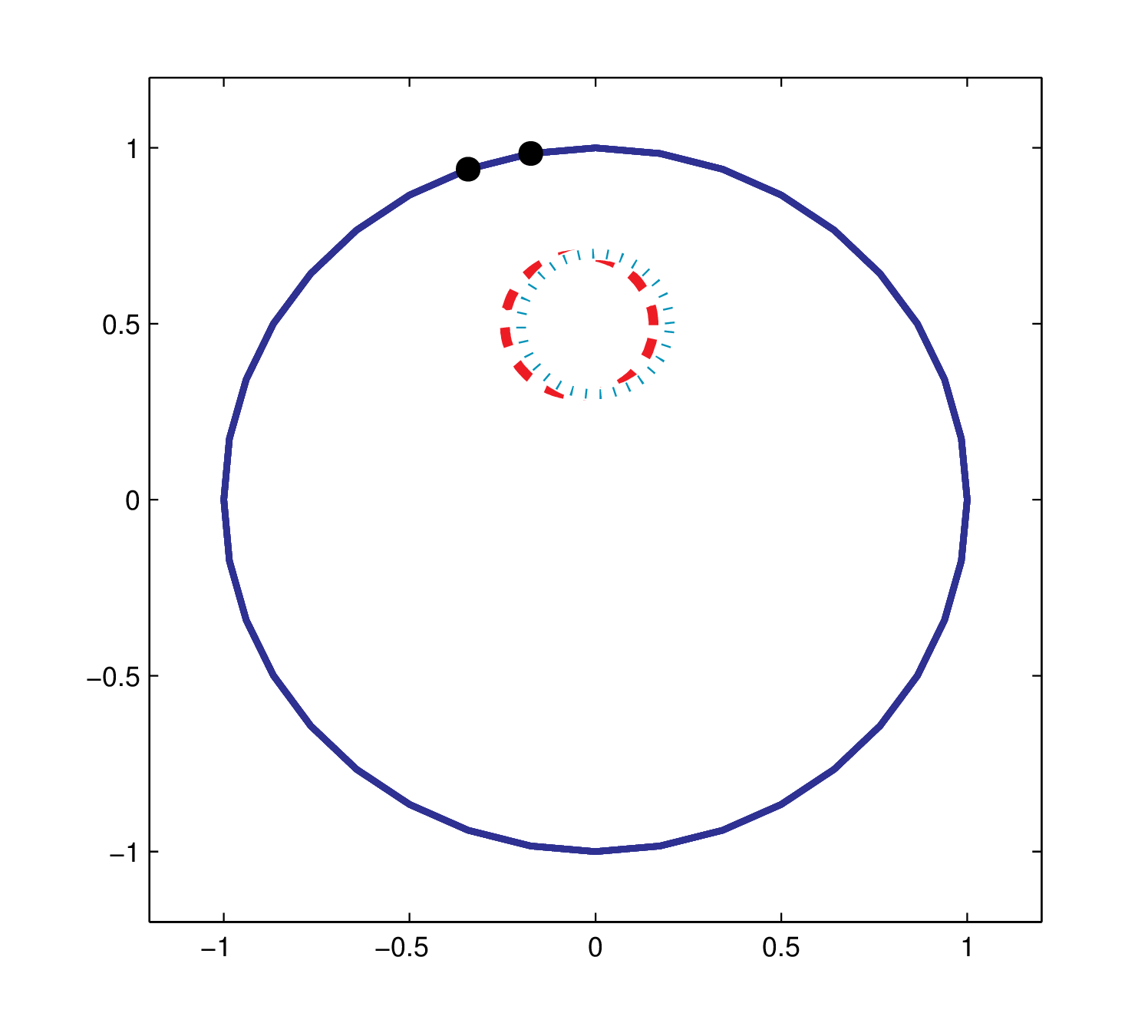}}
  \subfigure[]{
  \includegraphics[width=0.31\textwidth,height=1.9in]{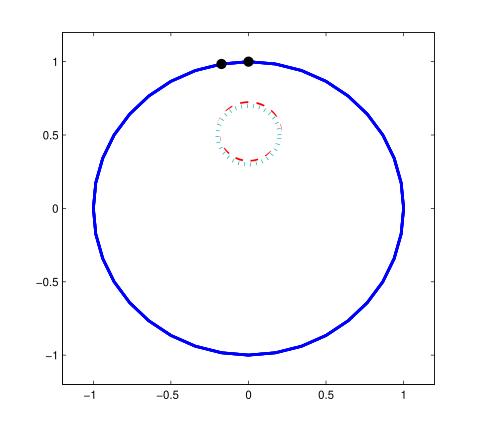}}  
 \caption{Estimates of $\bm \eta$ using posterior sample mean. The center of the blue dashed circle is the true $\bm \eta$, while the center of the red dashed circle is the average of all prediction samples. Flux measurements are obtained through two sensors (black dots) situated along the simulation boundary, with their positions dynamically adjusted over time according to Algorithm \ref{algo_pseudo_algo}. Specifically, the simulation is first performed to compute the flux at $T_1=0.5$, with sensors located at $(\frac{22}{18}\pi, \ \frac{30}{18}\pi)$. Subsequently, flux calculations are performed at $T_2=1$, with sensors positioned at $(\frac{10}{18}\pi, \ \frac{11}{18}\pi)$, and at $T_3=1.5$ with sensor placement at $(\frac{10}{18}\pi, \ \frac{9}{18}\pi)$. For further details on the sensor migration strategy, refer to Section \ref{sec_sensor_moving_peanut} } \label{Circle}
\end{figure}

\begin{figure}[h!]
  \centering
  \subfigure[]{
  \includegraphics[width=0.31\textwidth,height=1.9in]{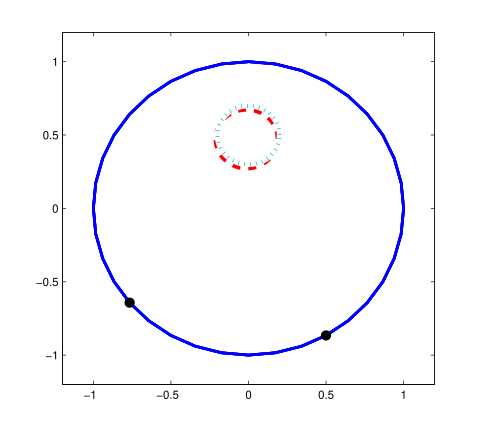}}
  \caption{Posterior mean estimations of $\bm \eta$ under static sensor positions. The sensors are initially positioned at $(\frac{22}{18}\pi, \ \frac{30}{18}\pi)$. In contrast to the dynamic sensor adjustment recommended by Algorithm \ref{algo_pseudo_algo}, we maintain a fixed sensor location and capture flux measurements at time instances $t'=[0.5, 1, 1.5]^{\intercal}$.} \label{Circle-fix}
\end{figure}

\subsubsection{Sensors migration rule}
\label{sec_sensor_moving_circle}
{
The parameters $\xi_1$ and $\xi_2$ represent the radius and center coordinates of the target circle. As the observation locations approach the target area, the gathered data becomes more informative. Since the target area is circular, the proximity of the observation locations to the domain can be quantified by measuring the distance between the center and the observation locations. Hence, the moving observation locations are determined based on the samples of $\xi_2$, specifically,
\begin{align}\label{rule of orbit1}
    \text{New sensor location} = (\lfloor {\bar \omega}/h_\theta\rfloor, \lceil  {\bar \omega}/h_\theta \rceil). 
\end{align}
where $\bar \omega$ is the posterior mean of $\omega(\xi_2)$, $h_\theta$ is the spatial step of the $\theta$ direction, yielding $\lfloor {\bar \omega}/h_\theta\rfloor$ and $\lceil  {\bar \omega}/h_\theta \rceil$ the index of locations along the boundary.}

Firstly, we take two measured locations along the edge randomly at time $T_1=0.5$, and draw samples from the posterior density. We then calculate the posterior mean of $\omega(\xi_2)$ as $\bar \omega$, and determine the new sensors location index as $\lfloor {\bar \omega}/h_\theta\rfloor$ and $\lceil  {\bar \omega}/h_\theta \rceil$ at the next time layer $T_2$. Obviously, the selected observation points $p_1$ and $p_2$ are adjacent. The process is repeated until the observation time is $\tilde t=T_3$, and the forward model simulation and Bayesian inference are then stopped. The results of the dynamic procedure are displayed in Figure \ref{Circle}.

For comparison purposes, we maintain the sensor location fixed at $(\frac{22}{18}\pi, \frac{30}{18}\pi)$, and flux measurements are taken at time instances $t'=[0.5, 1, 1.5]^{\intercal}$, respectively. {The measurement error remains consistent for both the fixed location strategy and the strategy involving sequentially determined locations. In other words, for each observed time $T_i$, the noise introduced to the flux at the fixed locations $(\frac{22}{18}\pi, \frac{30}{18}\pi)$ is identical to the noise introduced to the flux at the sequentially determined locations}. In Figure \ref{Circle-fix}, the posterior mean of the samples is represented by the red circle, offering an accurate estimate of the target region's center. However, Figure \ref{Circle-COM} reveals a distinct pattern: the variance of the samples acquired through sequentially determined sensor placements is significantly reduced compared to those obtained using the fixed measurement strategy. This outcome is a result of the fact that sequentially determined measurements provide more informative data, leading to a substantial reduction in uncertainty.

\begin{figure}[h!]
  \centering
  \subfigure[]{
  \includegraphics[width=0.45\textwidth,height=2.1in]{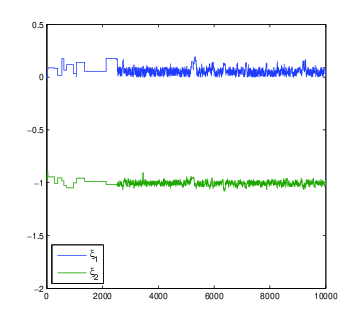}}
  \subfigure[]{
  \includegraphics[width=0.45\textwidth,height=2.1in]{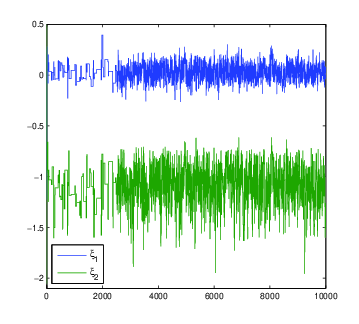}} 
  \caption{
  Trace plots illustrating posterior samples across distinct measurement strategies. The blue curve represents samples for $\xi_1$, while the green curve corresponds to $ \xi_2$.
(a) Utilizing sequentially determined locations as recommended by Algorithm \ref{algo_pseudo_algo}.
(b) Employing fixed sensors.
The visual analysis of the plots indicates that the strategy involving fixed sensor positions results in a broader confidence interval. Conversely, the proposed strategy yields predictions with significantly reduced variance. }\label{Circle-COM}
\end{figure}

We additionally performed a comparative analysis between our proposed method and a random selection strategy for measurement locations. To gain a more comprehensive understanding of the stochastic nature, we conducted two experiments with distinct measurement locations. The trace plots of posterior samples for $\bm \xi$ are presented in Figure \ref{Circle-COM2}. It is evident that the dynamical sensors placement method, utilizing both the proposed varying method as detailed in Section \ref{sec_sensor_moving_circle}, and the random varying strategy successfully capture the inverse Quantities of Interest (QoIs). Nevertheless, it is important to highlight that the random varying strategy results in a broader confidence interval.

\begin{figure}[h!]
 \centering
   \subfigure[]{
  \includegraphics[scale = 0.35]{graphs/Ctrace_seq.png}}
  \subfigure[]{
  \includegraphics[scale = 0.35]{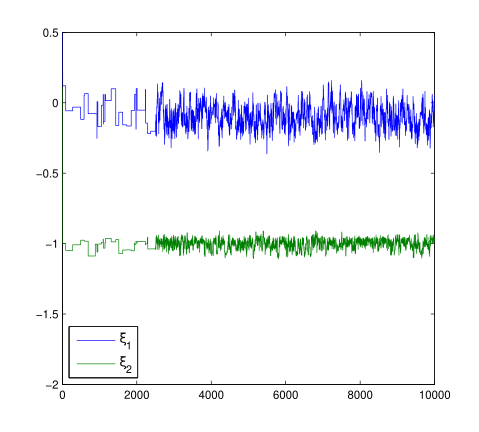}}
  \subfigure[]{
 \includegraphics[scale = 0.35]{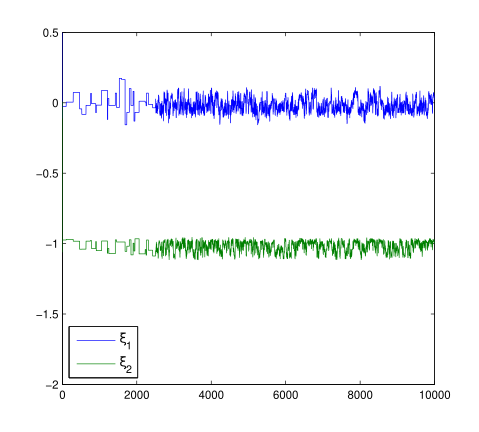}} 
 \caption{
 Trace plots illustrating posterior samples across distinct measurement strategies. The blue curve represents samples for $\xi_1$, while the green curve corresponds to $ \xi_2$.
Figure (a): utilizing sequentially determined locations as recommended by Algorithm \ref{algo_pseudo_algo} with \textcolor{black}{dynamical sensor placement strategy detailed in Section \ref{sec_sensor_moving_circle}}.
Figure (b) and (c): using Algorithm \ref{algo_pseudo_algo} with random sensor placement strategy; please note that \textcolor{black}{we repeat the experiments twice to comprehensively study the performance.}
It is evident that adopting a dynamic sensor placement strategy (left figure) leads to a narrower confidence band, signifying decreased variance among all the samples. Nevertheless, even randomly varying sensors, as illustrated in (b) and (c), still exhibit considerably lower variance compared to the fixed sensor approach, as demonstrated in Figure \ref{Circle-COM} (b).
}
\label{Circle-COM2}
\end{figure}

\subsection{High dimensional peanut-shape source}
In this example, we examine a problem of higher dimensions with asymmetrical characteristics \cite{RundellZhang:2018}.
Specifically, the unknown domain has the boundary $\partial D$ that can be parameterized as
\[
\partial D =\{q(\theta; \bm \xi)(\cos(\theta), \sin(\theta))^T: \theta \in [0, 2\pi] \},
\]
with a smooth, periodic function $q(\theta; \bm \xi)\in (0, 1)$, which has the form
\begin{align}
    q(\theta; \bm \xi)=\frac{1}{2}\xi_1+\sum_{i=1}^M\bigg( \xi_i\cos(i\theta)+\xi_{i+1}\sin(i\theta)\bigg).
    \label{eqn_peanut}
\end{align}
To ensure the smoothness of the approximation, we follow the paper \cite{RundellZhang:2018} and set the penalty term to be the $H^2$ norm of $q(\theta; \bm \xi)$, which implies the Gaussian prior density $ {\cal N}(0, B)$ for $\bm \xi$, where the covariance matrix $B\in {\mathbb R}^{(2M+1)\times (2M+1)}$ is diagonal, with entries
\[
B_{1,1}=1, \quad B_{i+1,i+1}=B_{i+M+1,i+M+1}=\frac{1}{i^2}, \quad i=1, \cdots, M.
\]

The true parameter is $\bm \xi=[1, 0, 0, 0, 0.3]^{\intercal}$, the strength is preset as $b=10$, the variance of the measurement error is $\sigma^2=0.01^2$. The iteration numbers are $N_1=1\times 10^3$, $N=1.5\times10^4$, the update frequent number is $k_0=2.5\times 10^3$. The parameters in the pCN and adaptive pCN scheme are tuned separately, both lead to the acceptance rate $30\%\sim 40\%$.

We first assess the random sensor relocation method and showcase the corresponding outcomes in Figure \ref{FPeanut-mean} (b).
We compare the proposed method with the fixed-sensor approach, and as shown in Figure \ref{FPeanut-mean} (a), the limited data available to the fixed-location strategy hinders its ability to accurately capture the shape of the target source $D$. 
\begin{figure}[h!]
  \centering
  \subfigure[ ]{
  \includegraphics[width=0.31\textwidth,height=1.9in]{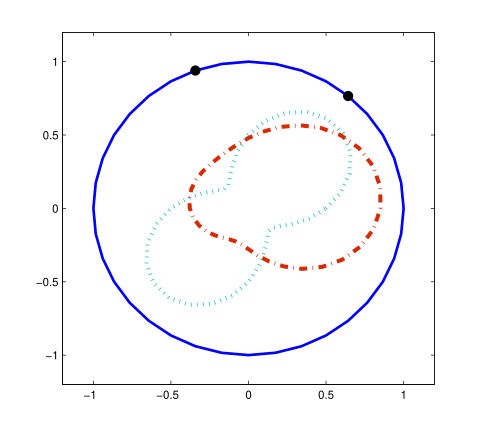}}
  \subfigure[ ]{
  \includegraphics[width=0.31\textwidth,height=1.9in]{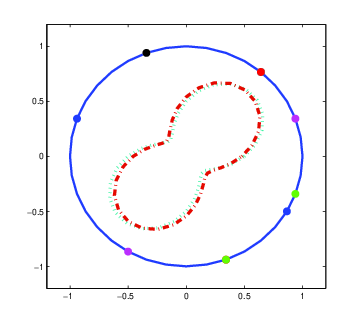}}
  \caption{Posterior mean estimates of $q(\theta,\bm \xi)$ for sensors at time $t= T_5$ with(a) fixed sensor locations $(\frac{11}{18}\pi, \ \frac{5}{18}\pi)$ and (b) random moving sensors at time: $T_1: (\frac{11}{18}\pi, \ \frac{5}{18}\pi)$, $T_2: (\frac{29}{18}\pi, \ \frac{5}{18}\pi)$, $T_3: (\frac{16}{18}\pi, \ \frac{33}{18}\pi)$, $T_4: (\frac{29}{18}\pi, \ \frac{34}{18}\pi)$, $T_5: (\frac{24}{18}\pi, \ \frac{2}{18}\pi)$. The shape of the high dimensional source is accurately captured by the moving sensors strategy, while the fixed locations strategy performs poorly.}
  \label{FPeanut-mean}
\end{figure}
Despite some overlap in measurement locations across different time points, the data collected by randomly relocating sensors proves to be sufficiently informative for capturing the intricate shape of the high-dimensional source. As depicted in Figure \ref{fig_FPeanut_trace}, the samples derived from the randomly moving sensors strategy tend to cluster around or in proximity to the inverse Quantities of Interest (QoIs) denoted as $\bm \xi$, whereas the fixed-sensor approach struggles to identify most of the QoIs, with the exception of $\xi_1$.
\begin{figure}[h!]
  \centering
  \subfigure[ ]{
  \includegraphics[width=0.45\textwidth,height=2.1in]{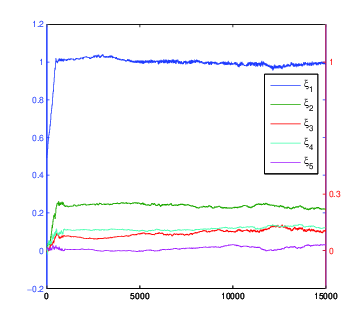}}
  \subfigure[ ]{
  \includegraphics[width=0.45\textwidth,height=2.1in]{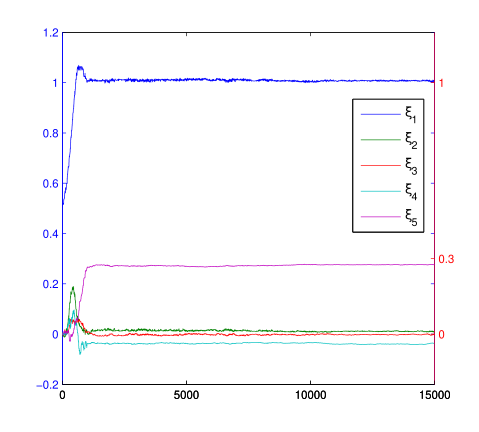}}
  \caption{Trace plots illustrating posterior samples at all iteration steps.  Different color represents different unknown parameters, with true parameters $\bm \xi=[1, 0, 0, 0, 0.3]^{\intercal}$.
  Left (a): 
 the sensors' location are fixed at $(\frac{11}{18}\pi, \ \frac{5}{18}\pi)$. Right (b): randomly move the sensors along the boundary. The samples of $\xi_1$ and $\xi_5$ concentrates around or near the true value $1$ and $0.3$ for the moving sensors strategy, the components $ \xi_2$, $ \xi_3$, and $\xi_4$ are also around the true value $0$.}
  \label{fig_FPeanut_trace}
\end{figure}

\subsubsection{Sensors migration rule}
\label{sec_sensor_moving_peanut}
The random sensor relocation strategy demonstrates superior performance compared to the fixed-sensor approach. However, it remains challenging for us to determine the optimal time to relocate the sensors. To address this issue, we introduce a second effective approach that can help us identify when it is appropriate to cease relocating the sensors.
 
Our approach revolves around relocating the two sensors to locations characterized by the highest flux variance observed across all sampling iterations. To be precise, let the set of flux values at position $\alpha$ be denoted as $\mathcal{F}_{\alpha} = \{f_\alpha^i \}_{k}$, where $k$ represents the index of the samples. We then proceed to shift one sensor to:
\begin{align}\label{rule of orbit}
    \text{New sensor location} = argmax_{\alpha} \text{var}(\mathcal{F}_\alpha). 
\end{align}
The second sensor is subsequently moved to the position exhibiting the second largest variance.
It's important to emphasize that this strategy incurs no additional costs.
To compute the flux at the sensor locations, we must evaluate the equation's solution using the sampled source, guided by the acceptance rate (\ref{acc-rate}). This enables us to compute the flux variance across all sampling iterations at any given point along the domain boundary. Notably, higher variance corresponds to a larger confidence interval and a less confident prediction. Consequently, our strategy involves relocating the sensors to positions characterized by the greatest variance. 

Following the aforementioned rule, we sequentially ascertain the measurement locations and present the posterior mean of the prediction samples in Figure \ref{Peanut}. Notably, the locations at time $T_5$ coincide with those at time $T_3$. Specifically, the suggested locations at time $T_5$ are $(\frac{21}{18}\pi, \ \frac{20}{18}\pi)$. This recurrence of the location $\frac{20}{18}\pi$ also appears at $T_4$, indicating a pattern of repeated cycling between the locations $(\frac{4}{18}\pi, \ \frac{5}{18}\pi)$ and $(\frac{19}{18}\pi, \ \frac{20}{18}\pi)$. Consequently, we choose to conclude the sequential process at time $T_5$. This method offers a means to determine when to cease taking measurements, providing insights into the richness of the collected data to some extent.
The subfigure includes the $1.5 \sigma$ confidence interval of the posterior samples for $q(\theta; \bm \xi)$, within which the actual $q(\theta; \bm \xi)$ value is predicted.
\begin{figure}[h!]
  \centering
  \subfigure[]{
  \includegraphics[width=0.31\textwidth,height=1.9in]{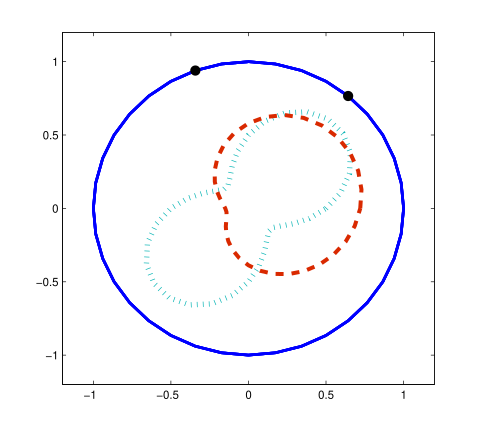}}
  \subfigure[]{
  \includegraphics[width=0.31\textwidth,height=1.9in]{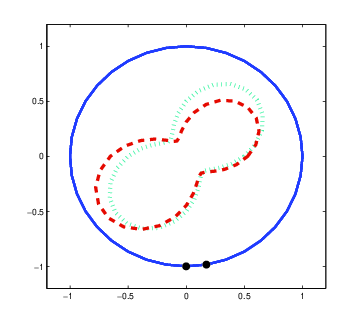}}
  \subfigure[]{
  \includegraphics[width=0.31\textwidth,height=1.9in]{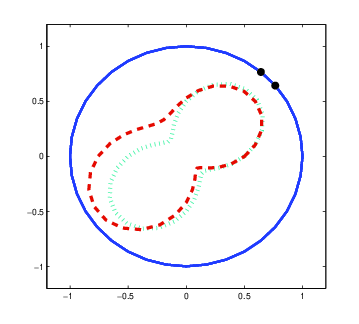}}
  \subfigure[]{
  \includegraphics[width=0.31\textwidth,height=1.9in]{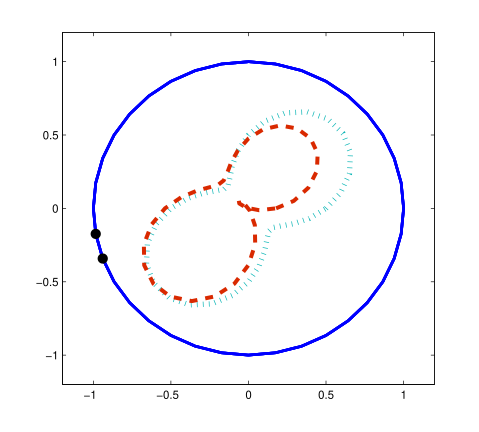}}
  \subfigure[]{
  \includegraphics[width=0.31\textwidth,height=1.9in]{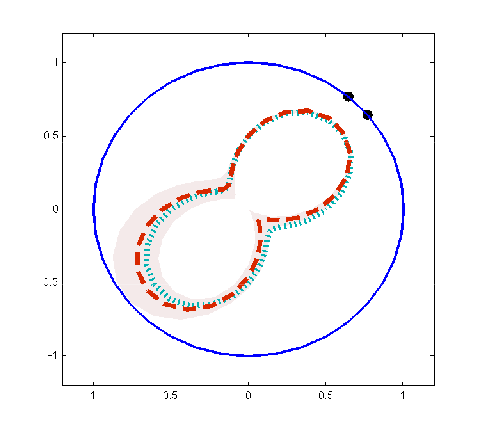}}
  \caption{Posterior mean estimates for prediction samples by the proposed algorithm \ref{algo_pseudo_algo}. The light blue curves represent the target source location defined in Equation \ref{eqn_peanut}, while the red dashed line depicts the sample mean. Black dots situated along the boundary denote the sensor positions, and the detailed determination of sensor locations at different time layers is presented sequentially below. (a) $T_1=0.5$, $(\frac{11}{18}\pi, \ \frac{5}{18}\pi)$. (b) $T_2=1$, $(\frac{28}{18}\pi, \ \frac{27}{18}\pi)$. (c) $T_3=1.5$, $(\frac{4}{18}\pi, \ \frac{5}{18}\pi)$. (d) $T_4=2$, $(\frac{19}{18}\pi, \ \frac{20}{18}\pi)$. (e) $T_5=2.5$, $(\frac{4}{18}\pi, \ \frac{5}{18}\pi)$. The shaded band is plotted with $\mu\pm 1.5\sigma$.  }
  \label{Peanut}
\end{figure}

\section{Concluding remarks}
In this work, we consider the inverse source problem of the heat equation. We use the boundary measurements where the observation sensors can be moved. The Bayesian method is used and we attempt some numerical examples. The numerical results illustrate that the Bayesian approach is feasible to solve such an inverse problem which uses the data from the moving observation points. 

We outline several directions for future research.
First, it is worth exploring the application of the PDE \eqref{PDE} in a more general domain. In this study, we confined our attention to the unit disc in $\mathbb R^2$, which possesses advantageous geometric properties. Expanding our analysis to encompass general domains in $\mathbb R^2$ or $\mathbb R^3$ would undoubtedly introduce additional complexities to the uniqueness argument and reconstruction algorithms.

Second, a significant open question pertains to the optimal strategy for relocating observation sensors. While we proposed a migration rule \eqref{rule of orbit} in this work, we acknowledge that it lacks rigorous mathematical proofs, rendering it somewhat empirical in nature. Consequently, an essential avenue for future research involves a thorough investigation of this migration rule to establish its validity and optimize its performance.

\section*{Acknowledgements.}
Na Ou acknowledges the support of Chinese NSF 11901060, Hunan Provincial NSF 2021JJ40557 and Scientific Research Foundation of Hunan Provincial Education Department 22B0333. Zhidong Zhang is supported by National Natural Science Foundation of China (Grant No. 12101627). Guang Lin acknowledges the support of the National Science Foundation (DMS-2053746, DMS-2134209, ECCS-2328241, and OAC-2311848), and U.S. Department of Energy (DOE) Office of Science Advanced Scientific Computing Research program DE-SC0021142 and DE-SC0023161.

\smallskip
\bigskip

\bibliographystyle{plain}
\bibliography{ref}

\end{document}